\documentclass[a4paper,11pt]{article}
\usepackage[utf8]{inputenc}
\usepackage{amsfonts}
\usepackage{enumitem}
\usepackage{amssymb, amsmath, amsfonts, amsthm, fancyhdr}
\usepackage[left=25mm,right=25mm,top=25mm,bottom=20mm,marginparwidth=20mm,marginparsep=5mm]{geometry}
\usepackage{parskip}
\usepackage{graphicx}
\usepackage{verbatim}

\DeclareMathOperator{\eps}{\varepsilon}

\newtheorem{theorem}{Theorem}

\newtheorem{lemma}{Lemma}
\newtheorem{claim}{Claim}
\newtheorem{proposition}{Proposition}
\newcommand{\seq}[3]{#1_{1} #3 #1_{2} #3 \ldots #3 #1_{#2}}

\title{Polynomial removal lemma for ordered matchings}
\date{}
\author{
Lior Gishboliner
\thanks{Department of Mathematics, ETH, Z\"urich, Switzerland. Research supported in part by SNSF grant 200021\_196965. Email: lior.gishboliner@math.ethz.ch.}
\and 
Borna Šimić\thanks{Department of Mathematics, ETH, Z\"urich, Switzerland. Email: bosimic@student.ethz.ch.}
}

\begin{document}
% \lhead{\fontfamily{lmss}\selectfont Polynomial dependence for hypermatchings}
% \pagenumbering{gobble}
% \rhead{\fontfamily{lmss}\selectfont Borna Šimić}
% \pagenumbering{gobble}

\maketitle
\begin{abstract}
    We prove that for every ordered matching $H$ on $t$ vertices, if an ordered $n$-vertex graph $G$ is $\varepsilon$-far from being $H$-free, then $G$ contains $\text{poly}(\varepsilon) n^t$ copies of $H$. This proves a special case of a conjecture of Tomon and the first author. We also generalize this statement to uniform hypergraphs. 
\end{abstract}

% \begin{flushleft} 
\section{Introduction}
The graph removal lemma is a fundamental result in extremal graph theory, stating that for every fixed graph $H$ and $\varepsilon > 0$, if an $n$-vertex graph $G$ is $\varepsilon$-far from being $H$-free, in the sense that $\varepsilon n^2$ edges must be deleted in order to turn $G$ into an $H$-free graph, then $G$ contains at least $\delta n^{|V(H)|}$ copies of $H$, where $\delta = \delta(H,\varepsilon) > 0$. This was proved in a seminal work of Ruzsa and Szemer\'edi \cite{RS}. The removal lemma was subsequently generalized to many other combinatorial structures, notably induced subgraphs \cite{AFKS}, hypergraphs \cite{NRS,Rodl_Skokan_2,Tao} and ordered graphs \cite{ABF}. Removal lemmas are also closely related to graph property testing in the dense graph model, where they correspond to testing algorithms with constant query complexity, see the book \cite{Goldreich}. 

A drawback of the known proofs of the removal lemma (and its many generalizations) is that all such proofs rely on Szemer\'edi's regularity lemma \cite{Szemeredi} or a generalization thereof. This results in weak quantitative bounds; for example, for the graph removal lemma stated above, the best known bound \cite{Fox} is that $1/\delta \leq \text{tower}(O(\log \left(1/\varepsilon\right)))$, where $\text{tower}(x)$ is a tower of $x$ exponents. This situation has led to research on the problem of characterizing the cases where the removal lemma has polynomial bounds, namely, where $\delta$ depends polynomially on $\varepsilon$. By now there are many works of this type \cite{Alon,AB,AFN,AS_digraphs,AS_induced,GS_graph_families,GS_C4,GS_poly_hypergraph_removal,GT_hypergraph,GT_ordered}. 

Here we focus on ordered graphs. 
An ordered graph is a graph with a linear ordering on its vertices. A copy of an ordered graph $H$ in an ordered graph $G$ is an injection $\varphi : V(H) \rightarrow V(G)$ which preserves the vertex order and satisfies that $\varphi(x)\varphi(y) \in E(G)$ for every $xy \in E(H)$.

In an important work \cite{ABF}, Alon, Ben-Eliezer and Fischer proved an ordered analogue of the graph removal lemma. They further asked to study cases where the ordered removal lemma has polynomial bounds. Addressing this question, Tomon and the first author \cite{GT_ordered} characterized the ordered graphs $H$ for which the induced $H$-removal lemma has polynomial bounds. They also studied the non-induced case, and conjectured that the (non-induced) $H$-removal lemma has polynomial bounds if and only if the ordered core\footnote{The ordered core of $H$ is defined as follows. Recall that a graph homomorphism from a graph $G$ to a graph $G'$ is a map $\varphi : V(G) \rightarrow V(G')$ such that $\varphi(x)\varphi(y) \in E(G')$ for every $xy\in E(G)$. For ordered graphs $G,G'$, an ordered homomorphism from $G$ to $G'$ is a graph homomorphism which also preserves the vertex order. The ordered core of $G$ is defined as the smallest (in terms of number of vertices) subgraph of $G$ such that there is an homomorphism from $G$ to $G'$. One can show that the ordered core is unique up to ordered isomorphism.} of $H$ is an (ordered) forest. As observed in \cite{GT_ordered}, to prove this conjecture it suffices to show that the $H$-removal lemma has polynomial bounds for every ordered forest $H$. Here we make progress on this conjecture by proving it for every ordered matching $H$. We also generalize this to $s$-uniform hypergraphs, $s \geq 3$. 
An (ordered) $n$-vertex $s$-uniform hypergraph $G$ is said to be $\varepsilon$-far from being $H$-free if one has to delete at least $\varepsilon n^s$ edges to turn $G$ into an $H$-free hypergraph. Our main result is as follows.  

\begin{theorem}\label{thm:main}
For every $t \geq s \geq 2$, there exists $C = C(t)$ such that the following holds. 
Let $\varepsilon > 0$, let $H$ be an ordered $s$-uniform matching on $t$ vertices, and let $G$ be an ordered $s$-uniform hypergraph on $n$ vertices. If $G$ is $\varepsilon$-far from being $H$-free, then $G$ contains at least $(\varepsilon/C)^C n^t$ copies of $H$.  
% There exists a polynomial $P$ dependent only on $H$ such that if $G$ is $\varepsilon$-far from $H$-freeness, $G$ contains at least $P(\varepsilon)n^t$ copies of $H$.
\end{theorem}

Proving removal-type statements for ordered structures tends to be considerably more difficult than for their unordered counterparts. For example, the proof of the ordered removal lemma in \cite{ABF} is substantially more involved than the original proof of the removal lemma in \cite{RS}. The key difficulty is to find copies of $H$ which respect the vertex order. To deal with this difficulty, our proof uses a novel argument of considering nested partitions of the vertex-set (with each partition refining the previous one) and ``cleaning" the graph with respect to each of these levels. 

\section{Proof of Theorem \ref{thm:main}}

Assume $G$ is as in the statement of the theorem, with vertex set $[n]$. 
For two subsets $A,B \subseteq [n]$, we write $A < B$ to mean that $a<b$ for all $a \in A, b\in B$, i.e., all elements of $A$ are smaller than all elements of $B$. 
We begin by partitioning\footnote{Here and throughout the proof, for the sake of simplicity, we omit floor and ceiling signs by assuming that $n$ is divisible by an appropriate (polynomial) function of $\varepsilon$.} $[n]$ into $k := \frac 1\varepsilon$ intervals $I_1 < \dots < I_k$ of length $\varepsilon n$ each, and delete all edges with at least two vertices inside one of these intervals. Let $G_0 \subseteq G$ be the resulting hypergraph. This step deletes less than \[\frac1\varepsilon{\varepsilon n \choose 2} \cdot {n-2 \choose s-2} < \frac{\varepsilon}{2} n^s\] edges, so $G_0$ is still $\frac{\varepsilon}{2}$-far from being $H$-free. 

Set $\gamma := \frac{\varepsilon}{4t}$ (recall that $t = |V(H)|$). 
For each $1 \leq \ell \leq k$, we define $t$ nested partitions of $I_\ell$ as follows. 
Set $\mathcal{J}_{\ell,1} = \{I_{\ell}\}$. For $j = 2,\dots,t$ and for each $J \in \mathcal{J}_{\ell,j-1}$, split $J$ into intervals of length $\gamma|J|$ and add these intervals to $\mathcal{J}_{\ell,j}$. Note that for each $1 \leq j \leq t$, $\mathcal{J}_{\ell,j}$ forms a partition of $I_{\ell}$ into intervals of size $\gamma^{j-1}|I_{\ell}|$. Put $\mathcal{J}_{\ell} := \bigcup_{j = 1}^t \mathcal{J}_{\ell,j}$.
For each $v \in I_{\ell}$ and $1 \leq j \leq t$, it will be convenient to denote by $J_j(v)$ the interval in $\mathcal{J}_{\ell,j}$ containing $v$, so that
\[v \in J_t(v) \subseteq J_{t-1}(v) \subseteq \ldots \subseteq J_1(v) = I_{\ell}.\] 
Set $\beta := 2\gamma = \frac{\varepsilon}{2t}$. 
% Take $\beta > 0$, also to be fixed later. 
We now perform a sequence of $k$ cleaning steps and define $k$ corresponding hypergraphs $G_0 \supseteq G_1 \dots \supseteq G_k$, where $G_{\ell}$ is the hypergraph obtained after the $\ell$th cleaning step ($1 \leq \ell \leq k$). 
At step $\ell$ we clean with respect to the interval $I_{\ell}$, as follows:
For every choice of $s-1$ vertices $v_1 < v_2 < \ldots < v_{s-1}$ outside of $I_{\ell}$, and for every interval $J \in \mathcal{J}_{\ell}$, 
let $L_{\ell}(v_1, v_2, \ldots v_{s-1}, J)$ denote the leftmost $\beta|J|$ vertices $w \in J$ such that $\{v_1,\dots,v_{s-1},w\} \in E(G_{\ell-1})$, if there are at least $\beta |J|$ such vertices, and else let $L_{\ell}(v_1, v_2, \ldots v_{s-1}, J)$ be the set of all such vertices $w$. Delete all edges $\{v_1,\dots,v_{s-1},w\} \in E(G_{\ell-1})$ with $w \in L_{\ell}(v_1, v_2, \ldots v_{s-1}, J)$. The resulting hypergraph is $G_{\ell}$.
By definition, for every given $(s-1)$-tuple $\seq v{s-1},$ and for every interval $J \in \mathcal{J}_{\ell}$, this operation deletes at most $\beta|J|$ edges of the form $\{v_1,\dots,v_{s-1},w\}$ with $w \in J$. Since the intervals in $\mathcal{J}_{\ell,j}$ form a partition of $I_{\ell}$ (for every $1 \leq j \leq t$), we delete at most $\beta|I_{\ell}|$ edges when considering these intervals. Summing over $1 \leq j \leq t$, this gives a total of at most $t\beta|I_{\ell}|$ edge deletions for each of the less than $n^{s-1}$ choices of $v_1,\dots,v_{s-1}$. Therefore,
% which adds up to a total of less than $t\beta n^{s-1} |I_{\ell}|$ edge deletions when obtaining $G_{\ell}$ from $G_{\ell-1}$. 
$e(G_{\ell-1}) - e(G_{\ell}) < t\beta n^{s-1} |I_{\ell}|$.
Summing over $\ell = 1,\dots,k$, we get that
\[ 
e(G_0) - e(G_k) < 
\sum_{\ell = 1}^k t\beta n^{s-1} |I_{\ell}| = t\beta n^s = \frac{\varepsilon}{2}n^s,
\] 
using our choice of $\beta$. 
As $G_0$ is $\frac{\varepsilon}{2}$-far from being $H$-free, $G_k$ must contain a copy of $H$. We will use this fact later on. 
% We then have that $G_k$ still contains a copy of $H$, as $G_0$ is $\frac{\varepsilon}{2}$-far from being $H$-free. 
% Fix $\beta = \frac {\varepsilon(s-1)!}{2t}$ and note that as we have a total of less than \[\frac{\varepsilon}{2(s-2)!} n^s +  \frac{t}{(s-1)!}\cdot \frac {\varepsilon(s-1)!}{2t} \cdot n^s < \varepsilon n^s\] edge deletions, $G_k$ still contains a copy of $H$. From now on, we fix $\gamma = \frac\beta2 = \frac {\varepsilon(s-1)! }{4t}$.
 The key property guaranteed by the cleaning procedure is the following.

 \begin{lemma}\label{lem:main}
     Let $1 \leq \ell \leq k$ and $1 \leq m \leq t$, let $w_1 < \dots < w_m$ be vertices in $I_{\ell}$, and let $v_{i, j} \in [n] \setminus I_{\ell}$, where $1 \leq i \leq m$ and $1 \leq j \leq s-1$, such that $\{v_{i, 1}, v_{i, 2}, \ldots v_{i, s-1} ,w_i\} \in E(G_{\ell})$ for every $i = 1,\dots,m$. Then there are at least $\left( \frac{\varepsilon}{4t} \right)^{m(t+1)}n^m$ different $m$-tuples of vertices $w'_1 < \nolinebreak \dots < \nolinebreak w'_m$ in $I_{\ell}$, such that $\{v_{i, 1}, v_{i, 2}, \ldots v_{i, s-1} ,w'_i\} \in E(G_{\ell-1})$ for every $i = 1,\dots,m$.
 \end{lemma}
 Note that by Lemma \ref{lem:main}, if there is a copy $H_{\ell}$ of $H$ in $G_{\ell}$ having $m$ vertices in $I_{\ell}$, then in $G_{\ell-1}$ there are at least $\left( \frac{\varepsilon}{4t} \right)^{m(t+1)}n^m$ copies of $H$ which agree with $H_{\ell}$ on the vertices outside of $I_{\ell}$.
 \begin{proof}[Proof of Lemma \ref{lem:main}]
Define sets $L_1,\dots,L_m \subseteq I_{\ell}$ as follows: 
$$L_i := L_{\ell}(v_{i, 1}, v_{i, 2}, \ldots, v_{i, s-1} ,J_{i}(w_i)) \setminus J_{i+1}(w_i)$$
for $1 \leq i < m$, and 
$$L_m := L_{\ell}(v_{m, 1}, v_{m, 2}, \ldots, v_{m, s-1},J_m(w_m)).$$ 
For every $1 \leq i \leq m$ and $w'_i \in L_i$, it holds that $\{v_{i, 1}, v_{i, 2}, \ldots v_{i, s-1} ,w'_i\} \in E(G_{\ell-1})$, by the definition of $L_{\ell}(v_{i, 1}, v_{i, 2}, \ldots, v_{i, s-1} ,J_{i}(w_i))$.
\begin{claim}\label{claim:order}
    $L_1 < \dots < L_m$.
\end{claim}
\begin{proof}
    Fix any $1 \leq i < m$, and let us show that $L_i < L_{i+1}$. 
    % We have $L_i \subseteq J_{i}(w_i)$ and $L_{i+1} \subseteq J_{i+1}(w_{i+1}) \subseteq J_{i}(w_{i+1})$. If $J_{i}(w_i) \neq J_{i}(w_{i+1})$ then $J_{i}(w_i) < J_{i}(w_{i+1})$ because $w_i < w_{i+1}$, and then clearly $L_i < L_{i+1}$.
    % We have $L_i \subseteq J_{i}(w_i)$ and $L_{i+1} \subseteq J_{i+1}(w_{i+1})$. 
    % So suppose that $J_{i}(w_i) = J_{i}(w_{i+1})$. 
    We begin by proving that $L_i < w_i$, namely, that all vertices of $L_i$ are to the left of $w_i$. Indeed, first note that $L_i \subseteq L_{\ell}(v_{i, 1}, v_{i, 2}, \ldots, v_{i, s-1},J_{i}(w_i))$.
    Recall that $L_{\ell}(v_{i, 1}, v_{i, 2}, \ldots, v_{i, s-1},J_{i}(w_i))$ is a set of leftmost vertices $w \in J_{i}(w_i)$ satisfying $\{v_{i, 1}, v_{i, 2}, \ldots v_{i, s-1},w\} \in E(G_{\ell-1})$, and the edges $\{v_{i, 1}, v_{i, 2}, \ldots v_{i, s-1},w\}$ for $w \in L_{\ell}(v_{i, 1}, v_{i, 2}, \ldots, v_{i, s-1},J_{i}(w_i))$ are deleted 
    when obtaining $G_{\ell}$ from $G_{\ell-1}$, while the edge $\{v_{i, 1}, v_{i, 2}, \ldots v_{i, s-1} ,w_i\}$ is still present in $G_{\ell}$. This shows that $L_{\ell}(v_{i, 1}, v_{i, 2}, \ldots, v_{i, s-1},J_{i}(w_i))$ is to the left of $w_i$, implying that $L_i < w_i$.  
    
    Next, note that $L_i$ is disjoint from $J_{i+1}(w_i)$ by definition. It follows that $L_i < J_{i+1}(w_i)$; indeed, for each $w \in L_i$, we have $J_{i+1}(w) \leq J_{i+1}(w_i)$ as $w < w_i$, and also
    $J_{i+1}(w) \neq J_{i+1}(w_i)$ because $w \notin J_{i+1}(w_i)$, hence $w < J_{i+1}(w_i)$. We conclude that $L_i < J_{i+1}(w_i) \leq J_{i+1}(w_{i+1})$, using that $w_i < w_{i+1}$. As $L_{i+1} \subseteq J_{i+1}(w_{i+1})$, we get that $L_i < L_{i+1}$, as required. 
\end{proof}
\begin{claim}\label{claim:size}
    $|L_i| \geq \left(\frac {\eps}{4t}\right)^{t+1} n$ for all $1 \leq i \leq m$. 
    \end{claim}
    \begin{proof}
    Recall that for every $w \in I_{\ell}$, we have 
    \[|J_t(w)| = \gamma |J_{t-1}(w)| = \ldots = \gamma^{t-1}|J_1(w)| = \gamma^{t-1}|I_{\ell}| = \gamma^{t-1}\varepsilon n > \left(\frac{\varepsilon}{4t}\right)^tn.\] 
    Now, observe that $|L_{\ell}(v_{i, 1}, v_{i, 2}, \ldots, v_{i, s-1}, J_i(w_i))| \geq \beta |J_i(w_i)|$, because otherwise we would have deleted all edges of the form $\{v_{i, 1}, v_{i, 2}, \ldots, v_{i, s-1},w\}$ with $w \in J_i(w_i)$ when obtaining $G_{\ell}$ from $G_{\ell-1}$, but the edge $\{v_{i, 1}, v_{i, 2}, \ldots, v_{i, s-1},w_i\}$ is still present in $G_{\ell}$. 
    Using that $|J_{i+1}(w_i)| = \gamma |J_i(w_i)| = \frac{\beta}{2}|J_i(w_i)|$, we get by the definition of $L_i$ that
    \begin{align*} |L_i| &\geq |L_{\ell}(v_{i, 1}, v_{i, 2}, \ldots, v_{i, s-1}, J_i(w_i))| - |J_{i+1}(w_i)| \geq \frac{\beta}{2}|J_i(w_i)| \geq \left(\frac{\varepsilon}{4t}\right)^{t+1}n.
    %\\ &\geq \frac{\varepsilon(s-1)!}{4t}|J_i(w_i)| \\ &\geq \left(\frac{\varepsilon(s-1)!}{4t}\right)^{t+1}n
    \end{align*}
    % as $J_i(w_i) \supseteq J_t(w_i)$.
    \end{proof}
    We now complete the proof of Lemma \ref{lem:main}.
    As we saw above, for every $1 \leq i \leq m$ and $w'_i \in L_i$, it holds that $\{v_{i, 1}, v_{i, 2}, \ldots v_{i, s-1} ,w'_i\} \in E(G_{\ell-1})$.
    Thus, the lemma follows from Claims \ref{claim:order} and \nolinebreak\ref{claim:size}.
 \end{proof}

Recall that $G_k$ contains a copy of $H$; denote it $H_k$. We can now use this initial copy and Lemma \ref{lem:main} to construct the required number of distinct $H$-copies in $G_0$. This will be done in the following lemma. 
For $1 \leq \ell \leq k$, let $m_{\ell}$ be the number of vertices of $H_k$ in the interval $I_{\ell}$. For convenience, put $\delta := \left(\frac {\varepsilon}{4t}\right)^{t+1}$.
\begin{lemma}\label{lem:copy construction}
For every $\ell = k,\dots,0$, there are at least 
$(\delta n)^{m_{\ell+1} + \dots + m_k}$ copies of $H$ in $G_{\ell}$ which have $m_i$ vertices in $I_i$ for every $1 \leq i \leq k$, and have the same vertices as $H_k$ in $I_1 \cup \dots \cup I_{\ell}$. 
\end{lemma}
\begin{proof}
    The proof is by reverse induction on $\ell$. 
    The base case $\ell = k$ holds trivially, because $H_k$ is a copy of $H$ in $G_k$ satisfying the required properties, and $(\delta n)^0 = 1$. 
    For the induction step, let $0 < \ell \leq k$. By the induction hypothesis, there is a collection $\mathcal{H}_{\ell}$ of at least $(\delta n)^{m_{\ell+1} + \dots + m_k}$ copies of $H$ in $G_{\ell}$ which have $m_i$ vertices in $I_i$ for every $1 \leq i \leq k$, and have the same vertices as $H_k$ in $I_1 \cup \dots \cup I_{\ell}$. 
    If $m_{\ell} = 0$ then there is nothing to prove, so suppose that $m_{\ell} \geq 1$. 
    Fix any $H_{\ell} \in \mathcal{H}_{\ell}$. Note that every edge of $H_{\ell}$ touching $I_{\ell}$ has exactly one vertex in $I_{\ell}$, because every edge of $G_0$ has at most one vertex in each of the intervals $I_1,\dots,I_k$ (by the definition of $G_0$). Namely, every edge $e \in E(H_{\ell})$ with $e \cap I_{\ell} \neq \emptyset$ is of the form $\{v_1,\dots,v_{s-1},w\}$ with $w \in I_{\ell}$ and $v_1,\dots,v_{s-1} \in [n] \setminus I_{\ell}$. Let $w_1 < \dots < w_{m_{\ell}}$ be the vertices of $H_{\ell}$ in $I_{\ell}$. 
    For each $1 \leq i \leq m_{\ell}$, let $v_{i,1},\dots,v_{i,s-1} \in [n] \setminus I_{\ell}$ such that $\{v_{i,1},\dots,v_{i,s-1},w_i\} \in E(H_{\ell})$.
    By Lemma \ref{lem:main}, we can replace $w_1,\dots,w_{m_{\ell}}$ in $(\delta n)^{m_{\ell}}$ ways to obtain copies of $H$ in $G_{\ell-1}$. Doing this for different $H_{\ell},H'_{\ell} \in \mathcal{H}_{\ell}$ gives different copies of $H$, because $H_{\ell},H'_{\ell}$ differ on vertices outside $I_{\ell}$ (as they both agree with $H_k$ on $I_{\ell}$), and we do not change the vertices of $H_{\ell},H'_{\ell}$ which are outside $I_{\ell}$.
    Thus, doing the above for each $H_{\ell} \in \mathcal{H}_{\ell}$ gives the required $(\delta n)^{m_{\ell}}|\mathcal{H}_{\ell}| \geq (\delta n)^{m_{\ell} + \dots + m_k}$ copies of $H$ in $G_{\ell-1}$. This completes the induction step. 
\end{proof}
For $\ell = 0$, Lemma \ref{lem:copy construction} gives 
\[ (\delta n)^{m_1 + \dots + m_k} = (\delta n)^t = \left(\frac {\varepsilon}{4t}\right)^{t(t+1)} n^t\]
copies of $H$ in $G_0$ (and so in $G$), as required. 
This proves the theorem. 

\section{Concluding remarks}
We proved that ordered matchings admit a polynomial removal lemma. It would be interesting to extend this result to other families of ordered forests, such as ordered paths. 

Our proof of Theorem \ref{thm:main} shows that one can take $C(t) = O(t^2)$. It would be interesting to improve this to $C(t) = O(t)$.
For fixed $s$ and growing $t$, this would be tight, as shown by the following proposition:
\begin{proposition}\label{prop:lower bound}
    Let $H$ be an ordered $s$-uniform hypergraph with $t$ vertices and $m$ edges, let $\varepsilon > 0$ be small enough, and let $n \geq n_0(\varepsilon)$. Then there exists an $n$-vertex $s$-uniform ordered hypergraph $G$ which is $\varepsilon$-far from being $H$-free but contains only $O(n^t \varepsilon^m)$ copies of $H$. 
\end{proposition}
If $H$ is an ordered matching then $e(H) = \frac{t}{s}$, so 
the proposition implies that $C(t) \geq \frac{t}{s}$. We give a proof sketch of the proposition.
\begin{proof}[Proof sketch of Proposition \ref{prop:lower bound}]
Take $G \sim \mathbb{H}_s(n,\varepsilon)$ to be the random (binomial) ordered $s$-uniform hypergraph; namely, each of the $\binom{n}{s}$ potential edges is included in $G$ with probability $\varepsilon$, independently. The expected number of copies of $H$ in $G$ is 
$\Theta(n^t \varepsilon^m)$, because there are $\Theta(n^t)$ potential copies, and each appears with probability $\varepsilon^{m}$. Similarly, a given edge of $G$ appears in expectation in $O(n^{t-s} \varepsilon^{m-1})$ copies of $H$. Using Azuma's inequality (see e.g. \cite[Chapter 7]{Prob_Method}), one can show that w.h.p. every edge appears in $O(n^{t-s} \varepsilon^{m-1})$ copies of $H$, and the total number of copies of $H$ is $\Theta(n^t \varepsilon^m)$. Thus, one has to delete $\Omega\left( \frac{n^t \varepsilon^m}{n^{t-s} \varepsilon^{m-1}} \right) = \Omega(\varepsilon n^s)$ edges to destroy all copies of $H$, meaning that $G$ is $\Omega(\varepsilon)$-far from being $H$-free. 
\end{proof}

\bibliographystyle{siam} 
\bibliography{library}

% \end{flushleft}

\end{document}